\documentclass{birkjour}

 \newtheorem{theorem}{Theorem}[section]
 \newtheorem{corollary}[theorem]{Corollary}
 
 \newtheorem{proposition}[theorem]{Proposition}
 \theoremstyle{definition}
 
 \theoremstyle{remark}
 \newtheorem{remark}[theorem]{Remark}
 \newtheorem{convention}{Conventions}
  
 \numberwithin{equation}{section}

\newcommand{\C}{\mathbb{C}}
\newcommand{\ZZ}{\mathbb{Z}}

\newcommand{\QQ}{\mathbb{Q}}
\newcommand{\NN}{\mathbb{N}}
\newcommand{\PP}{\mathbb{P}}

\newcommand{\MM}{\mathcal M}

\newcommand{\pic}{\hbox{Pic}}

\newcommand{\ima}{\hbox{Im}}
\newcommand{\rom}{\romannumeral}

\newtheorem{nonumbering}{Theorem}

\newtheorem{nonumberingc}{Corollary}

\begin{document}

%-------------------------------------------------------------------------
% editorial commands: to be inserted by the editorial office
%
%\firstpage{1} \volume{228} \Copyrightyear{2004} \DOI{003-0001}
%\seriesextraline{This is the Concrete Title of this Book\br H.E. R and S.T.C. W, Eds.}
%\firstpage{1}

\title[A family of $K3$ surfaces having finite--dimensional motive]
 {A family of $K3$ surfaces having finite--dimensional motive}

\author[Robert Laterveer]
{Robert Laterveer}

\address{Institut de Recherche Math\'ematique Avanc\'ee,\\ CNRS - Universit\'e 
de Strasbourg,\\ 7 Rue Ren\'e Des\-car\-tes,\\ 67084 Strasbourg Cedex,\\ FRANCE.}

\email{robert.laterveer@math.unistra.fr}

\thanks{Thanks to Claire Voisin for having written the invaluable and inspirational monograph \cite{Vo}. Thanks to the participants of the Strasbourg 2014--2015 groupe de travail based on \cite{Vo}. Many thanks and kusjes to Yasuyo, Kai and Len, who provide excellent working conditions in Schiltigheim.}


\date{}

\begin{abstract} This short note contains an example of a $4$--dimensional family of $K3$ surfaces having finite--dimensional motive. Some consequences are presented, for instance the verification of a conjecture of Voisin (concerning $0$--cycles on the self--product) for $K3$ surfaces in this family.
\end{abstract}

\maketitle

\section{Intro}

The notion of finite--dimensional motive, developed independently by Kimura and O'Sullivan \cite{K}, \cite{An}, \cite{MNP} has given considerable new impetus to the field of algebraic cycles. To give but one example: thanks to this notion, we now know the Bloch conjecture is true for surfaces of geometric genus zero that are rationally dominated by curves. It thus seems worthwhile to find concrete examples of varieties that have finite--dimensional motive, this being (at present) one of the few means of arriving at a satisfactory understanding of Chow groups.

Let us consider complex projective $K3$ surfaces. In case the Picard number is $19$ or $20$, it is known the motive is finite--dimensional \cite{P}. In the cases of Kummer surfaces and $K3$ surfaces with a Shioda--Inose structure, the motive is again finite--dimensional (\cite{K}, resp. \cite[Remark 48]{thoughts}); these last two cases have Picard number at least $17$. There are further sporadic examples of even Picard number \cite[Corollary 2]{P}, \cite{LSY}. Proving that a general $K3$ surface has finite--dimensional motive remains elusive (not to say out of reach).

The aim of this short note is to propose another example of $K3$ surfaces with finite--dimensional motive. Our example is a $4$--dimensional family, so the generic member has Picard number $16$.

\begin{nonumbering}[=Theorem \ref{main}] Let $X$ be a $K3$ surface defined as a complete intersection of quadrics in $\PP^5$:
  \[   \begin{cases}  a_0 x_0^2 +a_1x_1^2+\cdots +a_5 x_5^2=0 &\\
                                 b_0 x_0^2 +b_1x_1^2+\cdots +b_5 x_5^2=0 &\\
                                 c_0 x_0^2 +c_1x_1^2+\cdots +c_5 x_5^2=0\ , &\\
                           \end{cases}\]
                 with $a_i, b_i, c_i\in \C, i=0,\ldots,5$.
                 
   Then $X$ has finite--dimensional motive.
\end{nonumbering}

The proof is based on the fact that the group $(\ZZ/2\ZZ)^4$ acts symplectically on $X$, and is inspired by \cite{GS}, where a systematic study is undertaken of $K3$ surfaces with a symplectic $(\ZZ/2\ZZ)^4$ action. Since (as convincingly argued in \cite{GS}) such $K3$ surfaces are somehow ``close'' to Kummer surfaces, it may be hoped that finite--dimensionality can eventually be proven for {\em all\/} $K3$ surfaces with a symplectic $(\ZZ/2\ZZ)^4$ action (cf. remark \ref{maybe}).

To further highlight the interest of finite--dimensionality, we list a few consequences. These are the truth of a weak version of the relative Bloch conjecture (corollary \ref{relbloch}), and of the Beauville--Voisin conjecture (conjecture \ref{BV}) for this class of $K3$ surfaces.
Another consequence (whose proof uses not only finite--dimensionality, but rather the fact that the motive is of abelian type, in the sense of \cite{V3}) is that these $K3$ surfaces verify a conjecture made by Voisin \cite{V9}:

\begin{nonumberingc}[=Corollary \ref{voisinconj}] Let $X$ be a $K3$ surface as in theorem \ref{main}. Let $a,a^\prime\in A^2_{hom}(X)_{}$ be two $0$--cycles of degree $0$. Then
  \[ a\times a^\prime=a^\prime\times a\ \ \hbox{in}\ A^4(X\times X)_{}\]
  (here the notation $a\times a^\prime$ is short--hand for the cycle class $(p_1)^\ast (a)\cdot (p_2)^\ast(a^\prime)\in A^{4}(X\times X)$, where $p_1, p_2$ denote projection on the first, resp. second factor.)
\end{nonumberingc}

\vskip 0.8cm

\begin{convention} In this note, the word {\sl variety\/} will refer to a quasi--projective irreducible algebraic variety over $\C$, endowed with the Zariski topology. A {\sl subvariety\/} is a (possibly reducible) reduced subscheme which is equidimensional. 

%{\bf All Chow groups will be with rational coefficients}: 
We will denote by $A^jX$ the Chow group of codimension $j$ algebraic cycles on $X$.
%with $\QQ$--coefficients; 
%for $X$ smooth of dimension $n$ the notations $A_jX$ and $A^{n-j}X$ will be used interchangeably. 
%In the rare cases we will have something to say about Chow groups with integral coefficients, we will indicate this by writing $A^\ast X_{\ZZ}$.
Chow groups with rational coefficients will be denoted
  $ A^j(X)_{\QQ}:=A^j(X)\otimes_{\ZZ} \QQ$.
The notation $A^j_{hom}X$, resp. $A^j_{AJ}X$ will be used to indicate the subgroups of homologically trivial, resp. Abel--Jacobi trivial cycles.

The notation $H^jX$ 
%(and $H_jX$) 
will be used to indicate singular cohomology $H^j(X,\QQ)$. 
%(resp. Borel--Moore homology $H_j(X,\QQ)$).

\end{convention}

\vskip 0.8cm

\section{Preliminary}

We refer to \cite{K}, \cite{An}, \cite{I}, \cite{J4}, \cite{MNP} for the definition of finite--dimensional motive. 
An essential property of varieties with finite--dimensional motive is embodied by the nilpotence theorem:

\begin{theorem}[Kimura \cite{K}]\label{nilp} Let $X$ be a smooth projective variety of dimension $n$ with finite--dimensional motive. Let $\Gamma\in A^n(X\times X)_{\QQ}$ be a correspondence which is numerically trivial. Then there is $N\in\NN$ such that
     \[ \Gamma^{\circ N}=0\ \ \ \ \in A^n(X\times X)_{\QQ}\ .\]
\end{theorem}

 Actually, the nilpotence property (for all powers of $X$) could serve as an alternative definition of finite--dimensional motive, as shown by a result of Jannsen \cite[Corollary 3.9]{J4}.
   Conjecturally, any variety has finite--dimensional motive \cite{K}. We are still far from knowing this, but at least there are quite a few non--trivial examples:
 
\begin{remark} 
The following varieties have finite--dimensional motive: varieties dominated by products of curves \cite{K}, $K3$ surfaces with Picard number $19$ or $20$ \cite{P}, surfaces not of general type with vanishing geometric genus \cite[Theorem 2.11]{GP}, Godeaux surfaces \cite{GP}, Catanese and Barlow surfaces \cite{V8}, certain surfaces of general type with $p_g=0$ \cite{PW}, Hilbert schemes of surfaces known to have finite--dimensional motive \cite{CM}, generalized Kummer varieties \cite[Remark 2.9(\rom2)]{Xu},
 3--folds with nef tangent bundle \cite{Iy} (an alternative proof is given in \cite[Example 3.16]{V3}), 4--folds with nef tangent bundle \cite{Iy2}, log--homogeneous varieties in the sense of \cite{Br} (this follows from \cite[Theorem 4.4]{Iy2}), certain 3--folds of general type \cite[Section 8]{V5}, varieties of dimension $\le 3$ rationally dominated by products of curves \cite[Example 3.15]{V3}, varieties $X$ with Abel--Jacobi trivial Chow groups (i.e. $A^i_{AJ}X_{\QQ}=0$ for all $i$) \cite[Theorem 4]{V2}, products of varieties with finite--dimensional motive \cite{K}.
\end{remark}

\begin{remark}
It is worth pointing out that all examples of finite-dimensional motives known so far happen to be in the tensor subcategory generated by Chow motives of curves (i.e., they are ``motives of abelian type'' in the sense of \cite{V3}). 
That is, the finite--dimensionality conjecture is still unknown for any motive {\em not\/} generated by curves (there are many such motives, cf. \cite[7.6]{D}).

For $K3$ surfaces, it is expected that the motive is of abelian type. (Indeed, the Kuga--Satake construction \cite{KS}, combined with the standard conjectures, shows the homological motive of a $K3$ surface is generated by curves. Kimura's finite--dimensionality conjecture then implies the same holds for the Chow motive.)
\end{remark}

\section{Main}

\begin{theorem}\label{main} Let $X$ be a $K3$ surface defined as a complete intersection of quadrics in $\PP^5$:
  \[   \begin{cases}  a_0 x_0^2 +a_1x_1^2+\cdots +a_5 x_5^2=0 &\\
                                 b_0 x_0^2 +b_1x_1^2+\cdots +b_5 x_5^2=0 &\\
                                 c_0 x_0^2 +c_1x_1^2+\cdots +c_5 x_5^2=0\ , &\\
                           \end{cases}\]
                 with $a_i, b_i, c_i\in \C, i=0,\ldots,5$.
                 
   Then $X$ has finite--dimensional motive. (Actually, $X$ even has motive of abelian type.)
   \end{theorem}
   
 \begin{proof} We may assume that $X$ is a generic member of the family (this follows from \cite[Lemma 3.2]{Vo}).
  As explained in \cite[Section 10.2]{GS}, the $K3$ surface $X$ admits a symplectic action of the group $G=(\ZZ/2\ZZ)^4$, given by
 the transformations of $\PP^5$ changing an even number of signs in the coordinates. To determine the quotient $X/G$, one notes (following loc. cit.) that the
 invariant polynomials under the action of $G$ are exactly $x_0^2, x_1^2,\ldots, x_5^2$ and $x_0x_1x_2x_3x_4x_5$. Denoting them by $y_0,y_1,\ldots,y_5, t$, there is the relation
   \[ t^2=\prod_{i=0}^5 y_i\ ,\]
   and so the quotient $X/G$ is a double cover of the plane given by the intersection of hyperplanes in $\PP^5$:
   \[   \begin{cases}  a_0 y_0 +a_1y_1+\cdots +a_5 y_5=0 &\\
                                 b_0 y_0 +b_1y_1+\cdots +b_5 y_5=0 &\\
                                 c_0 y_0 +c_1y_1+\cdots +c_5 y_5=0\ . &\\
                           \end{cases}\] 
                  The branch locus consists of $6$ lines meeting at $15$ points; these $15$ points correspond to $15$ nodes on the surface $X/G$. For $X$ generic, no $3$ of the $6$ lines intersect.     
                  Let
                  \[ Y\ \to\ X/G\]
                  denote a minimal resolution of singularities, so $Y$ is a $K3$ surface (of the type studied in \cite{Par}). 
                  %There exists a diagram
                %  \[ \begin{array}[c]{ccc}
                 %    \wt{X} &\to& X\\
                  %  \ \ \  \downarrow{\wt{p}}& &\ \ \downarrow{p}\\
                   % Y &\to& X/G\ ,\]
                  %  \end{array}\]
              %  where $\wt{X}\to X$ is a sequence of blow--ups.    
                 Thanks to the work of Voisin \cite{V11} and Huybrechts \cite{Hu}, we know that $G$ acts trivially on $0$--cycles, i.e.
                 \[  A^2_{hom}(X)_{\QQ}= A^2_{hom}(X)_{\QQ}^G  \ .   \]  
               This implies that the rational map $X\dashrightarrow Y$ induces an isomorphism
             \[  A^2_{hom}(X)_{\QQ}= A^2_{hom}(Y)_{\QQ}      \ .\]        
             Using this isomorphism of Chow groups, one proves (for instance as in \cite[Theorem 3.3 (\rom1)]{Ped}) that there is also an isomorphism of Chow motives
             \[  h(X)\cong h(Y)\ \ \hbox{in}\ \MM_{\rm rat}\ ,\]
            induced by the rational map $X\dashrightarrow Y$.
            But the $K3$ surface $Y$, being rationally dominated by a product of curves \cite[Section 3]{Par}, has finite--dimensional motive.
            \end{proof}                                

\begin{remark} Rather than referring to the very general \cite{V11}, in the specific case of theorem \ref{main} one can also establish directly, ``by hand'', that there is an isomorphism
  \[  A^2_{hom}(X)_{\QQ}\cong A^2_{hom}(Y)_{\QQ}\ . \]
 % (from which it readily follows that 
  %$t_2(X)\cong t_2(Y)$ and hence 
 % $h(X)\cong h(Y)$). 
 To do this, one can apply \cite[Lemma 2.3]{WHu}.
  \end{remark}

\begin{remark}\label{maybe} In the very nice article \cite{GS}, Garbagnati and Sarti study $K3$ surfaces with a $(\ZZ/2\ZZ)^4$ symplectic action (such as $X$ of theorem \ref{main}), and also $K3$ surfaces (such as $Y$ in the above proof) arising as desingularizations of quotients of a $K3$ surface under a symplectic $(\ZZ/2\ZZ)^4$ action. The families of $K3$ surfaces satisfying one of these properties are $4$--dimensional, so the general element is {\em not\/} a Kummer surface. Yet (as shown in loc. cit.) $K3$ surfaces in these families are somehow ``close to Kummer surfaces'', in that they retain many of the special properties of Kummer surfaces.

It would be interesting to try and extend theorem \ref{main} to other $4$--dimensional families considered in \cite{GS}. An obvious test case would be the family of Heisenberg invariant quartic surfaces \cite{E}, \cite[Section 10.1]{GS}. More ambitiously: can one somehow prove (perhaps using the existence of an Enriques involution \cite[Theorem 7.15]{GS}) finite--dimensionality for {\em all\/} families considered in \cite{GS} ?

\end{remark}

\section{Consequences}

This section contains some corollaries of theorem \ref{main}. A first corollary is that an old conjecture of Voisin \cite{V9} is true for $X$:

\begin{corollary}\label{voisinconj} Let $X$ be a $K3$ surface as in theorem \ref{main}. Let $a,a^\prime\in A^2_{hom}(X)_{}$ be two $0$--cycles of degree $0$. Then
  \[ a\times a^\prime=a^\prime\times a\ \ \hbox{in}\ A^4(X\times X)_{}\]
  (here the notation $a\times a^\prime$ is a short--hand for the cycle class $(p_1)^\ast (a)\cdot (p_2)^\ast(a^\prime)\in A^{4}(X\times X)$, where $p_1, p_2$ denote projection on the first, resp. second factor.)
\end{corollary}

\begin{proof} (This is not a corollary of finite-dimensionality as such, but rather of the isomorphism $A^2_{hom}(X)\cong A^2_{hom}(Y)$ obtained in the proof of theorem \ref{main}, plus the fact that $Y$ is a particularly well--understood $K3$ surface.)

Since $A^2_{hom}(X)$ is torsion free \cite{R}, it suffices to prove the corresponding statement with rational coefficients. Using \cite[Lemma 3.2]{Vo}, we may assume $X$ is a generic member of the family of theorem \ref{main}.
%This follows from the isomorphism
%  \[  A^2_{hom}(X)_{\QQ}\cong A^2_{hom}(Y)_{\QQ}\ \]
%  (established in the course of the proof of theorem \ref{main}), plus the fact that corollary \ref{voisinconj} is true for $Y$ (this is proven in \cite[Proposition 41]{thoughts}, and can also be proven in a direct fashion along the lines of \cite[Theorem 3.4]{V9}). The commutative diagram
Let $Y$ be the $K3$ surface of the proof of theorem \ref{main}.
There is a commutative diagram
      \[ \begin{array}[c]{ccc}
         A^2_{hom}(X)_{\QQ}\otimes A^2_{hom}(X)_{\QQ} &\to & A^4(X\times X)_{\QQ}\\
         \ \ \uparrow{\cong}&&\uparrow\\
         A^2_{hom}(Y)_{\QQ}\otimes A^2_{hom}(Y)_{\QQ} &\to & A^4(Y\times Y)_{\QQ}\\ 
         \end{array}\]
 (where the left vertical arrow is an isomorphism thanks to theorem \ref{main}). We are now reduced to proving Voisin's conjecture for $Y$, i.e.
 
 \begin{proposition}\label{6lines} Let $Y$ be a desingularization of the double cover of $\PP^2$ branched along $6$ lines in general position. Let $a,a^\prime\in A^2_{hom}(Y)_{}$ be two $0$--cycles of degree $0$. Then
  \[ a\times a^\prime=a^\prime\times a\ \ \hbox{in}\ A^4(Y\times Y)_{}\ .\]
  \end{proposition}
  
  \begin{proof} This is \cite[Proposition 41]{thoughts}, whose proof we reproduce. 
  This proof hinges on the fact that the Kuga--Satake construction for $Y$ is algebraic. That is, according to Paranjape \cite{Par} there exist an abelian variety $A$ of dimension $g$ and a correspondence $\Gamma^\prime\in A^2(Y\times A\times A)$ such that
   \[   (\Gamma^\prime)_\ast\colon\ \ T_Y\ \to\ H^2(A\times A)\]
   is an injection. Since homological and numerical equivalence coincide for surfaces and abelian varieties, it follows that there is also an injection
   \[  \Gamma^\prime\colon t_2(Y)\ \to\  h^2(A\times A)\ \ \hbox{in}\ {\mathcal M}_{num}\ ,\]
   where $t_2(Y)$ is the transcendental part of the motive of $Y$ in the sense of \cite{KMP}, and ${\mathcal M}_{num}$ denotes the category of motives modulo numerical equivalence. Composing with some Lefschetz operator, one also gets an injection
   \[ \Gamma\colon\ \ t_2(Y)\ \to\ h^{4g-2}(A\times A)\ \ \hbox{in}\ {\mathcal M}_{num}\ \]
   (here $\Gamma$ is the composition $L^{2g-2}\circ \Gamma^\prime$, where $L$ is an ample line bundle on $A\times A$).

   The category ${\mathcal M}_{num}$ being semi--simple \cite{J1}, this is a split injection, i.e. there exists a correspondence $\Psi\in A^{g+1}(A\times A\times Y)$ such that
  \[   \Psi\circ \Gamma=\hbox{id}\colon\ \ t_2(Y)\ \to\ t_2(Y)\ \ \hbox{in}\ {\mathcal M}_{num}\ .\]
  But the motive $t_2(Y)$ is finite--dimensional (it is a direct summand of $h(Y)$, which is finite--dimensional since $Y$ is dominated by a product of curves \cite{Par}). This implies that there exists $N$ such that
    \[   \bigl( \Delta - \Psi\circ \Gamma \bigr)^{\circ N}=0\colon\ \ t_2(Y)\ \to\ t_2(Y)\ \ \hbox{in}\ {\mathcal M}_{rat}\ ,\]
    and hence that
    \[  \Gamma\colon\ \ A^2_{hom}(Y)_{\QQ}=A^2_{AJ}(Y)_{\QQ}= A^2_{AJ}(t_2(Y))_{\QQ}\ \to\ A^{2g}_{AJ}(A\times A)_{\QQ}\]
    is injective. 
    We note that, by construction, the action of $\Gamma$ on Chow groups factors as
    \[  \Gamma\colon\ \ A^2_{AJ}(Y)_{\QQ} \xrightarrow{ \Gamma^\prime} A^2(A\times A)_{\QQ} \xrightarrow{L^{2g-2}} A^{2g}(A\times A)_{\QQ}\ .\]
    Let $A^\ast_{(\ast)}()_{\QQ}$ denote Beauville's filtration on Chow groups of abelian varieties \cite{Beau}. It follows that
    \[ \Gamma_\ast \bigl( A^2_{AJ}(Y)_{\QQ}\bigr)\ \subset\ \bigoplus_{j\le 2} A^{2g}_{(j)}(A\times A)_{\QQ}\ ,\]
    as the Lefschetz operator preserves Beauville's filtration \cite{Kun}.
    On the other hand, 
    \[ \Gamma_\ast \bigl( A^2_{AJ}(Y)_{\QQ}\bigr)\ \subset\  A^{2g}_{AJ}(A\times A)_{\QQ}= \bigoplus_{j\ge 2} A^{2g}_{(j)}(A\times A)_{\QQ}\ .\]   
    The conclusion is that there is an injection
    \[ \Gamma_\ast\colon\ \ A^2_{AJ}(Y)_{\QQ}\ \to\ A^{2g}_{(2)}(A\times A)_{\QQ}\ .\] 
          
    The same argument gives also that
    \[  \Gamma\times\Gamma\colon\  \ima \Bigl( A^2_{hom}(Y)\otimes A^2_{hom}(Y)\ \to\ A^4(Y\times Y)\Bigr)\subset A^{4}(t_2(Y)\otimes t_2(Y)) \to A^{4g}(A^4)\]
    is injective. It now suffices to prove a statement for the abelian variety $B=A\times A$:
    
    \begin{proposition}\label{abcod2} Let $B$ be an abelian variety of dimension $2g$. Let
      \[ a, a^\prime\ \ \in A^{2g}_{(2)}(B)_{\QQ}\]
      be $2$ $0$--cycles. Then
      \[   a\times a^\prime - a^\prime\times a=0\ \ \hbox{in}\ A^{4g}(B\times B)_{\QQ}\ .\]
      \end{proposition}
      
      \begin{proof} The group $A^{2g}_{(2)}(B)_{\QQ}$ is generated by products of divisors
      \[   D_1\cdot D_2\cdot\ldots\cdot D_{2g}\ \ \in A^{2g}(B)_{\QQ}\ ,\]
      with $2$ of the $D_j$ in $A^1_{(1)}(B)_{\QQ}=\pic^0(B)_{\QQ}$, and the remaining $2g-2$ $D_j$ in $A^1_{(0)}(B)_{\QQ}$ \cite{Bl2}.
     As in \cite[Example 4.40]{Vo}, we consider the map
     \[ \sigma\colon B\times B\to B\times B,\ \ (a,b)\mapsto (a+b, a-b)\ .\]
     Since this is an isogeny, it induces an isomorphism on $A^\ast(B\times B)_{\QQ}$. But on the other hand,
     \[  \sigma\circ \iota \circ \sigma= 2 (\hbox{id}_B, -\hbox{id}_B)\colon\ \ B\times B\ \to\ B\times B\ .\]
     It thus suffices to note that
     \[  \begin{split}(\hbox{id}_B, -\hbox{id}_B)_\ast \bigl( &  D_1\cdot\ldots\cdot D_{2g}\times D_1^\prime\cdot\ldots\cdot D^\prime_{2g}\bigr)   =\\ &D_1\cdot\ldots\cdot D_{2g}\times D_1^\prime\cdot\ldots\cdot D^\prime_{2g}
         \  \  \ \hbox{in}\ \ A^{4g}(B\times B)_{\QQ}\ ,\\
         \end{split}\]
         since there is an even number of divisors $D_j^\prime$ for which $(-\hbox{id}_B)_\ast(D^\prime_j)=-D^\prime_j$ in $A^1(B)_{\QQ}$.
         \end{proof}
 \end{proof}
     \end{proof}

\begin{remark} Motivated by the Bloch--Beilinson philosophy, Voisin conjectures that corollary \ref{voisinconj} should be true for any regular surface with geometric genus $1$ \cite[page 270]{V9}. While this has been established in some special cases \cite{V9}, \cite{thoughts}, this conjecture is still wide open for a general $K3$ surface.
\end{remark}

Another corollary is that (a certain version of) the relative Bloch conjecture \cite{B}, \cite{V8} is true for this class of surfaces:

\begin{corollary}\label{relbloch} Let $X$ be a $K3$ surface as in theorem \ref{main}. Let $\Gamma\in A^2(X\times X)_{\QQ}$ be a correspondence such that
      \[ \Gamma_\ast\colon\ \ H^{2,0}(X)\ \to\ H^{2,0}(X)\ \]
      is the identity.
  Then 
  \[  \Gamma_\ast \colon\ \ A^2_{hom}(X)_{\QQ}\ \to\ A^2_{hom}(X)_{\QQ} \]
  is an isomorphism.
  \end{corollary}
  
  \begin{proof} 
  As is probably well--known, this follows from the finite--dimensionality; the proof goes as follows. The assumption implies (using an argument involving indecomposability of the Hodge structure $H^2(t_2(X))$ as in \cite[Corollary 3.11]{V8} or \cite[Lemma 2.5]{Ped}) that 
    \[  \Gamma_\ast=\hbox{id}\colon\ \ H^2(t_2(X))\ \xrightarrow{}\ H^2(t_2(X))\ \]
   (where $t_2$ denotes again the ``transcendental part of the motive'' as defined in \cite{KMP}). 
   It follows that
   \[  \Gamma-\Delta_X=\bigl( \Gamma-\Delta_X\bigr)\circ \bigl( \pi_0+\pi_2^{alg} +\pi_4\bigr)\ \ \in H^4(X\times X)\ ,\]
   where $\pi_0=x\times X$, $\pi_4=X\times x$, and $\pi_2^{alg}$ is the projector (constructed in \cite{KMP}) supported on $D\times D$, for some divisor $D$.
   It follows that we can write
   \[ \Gamma-\Delta_X= R_0+R_2+R_4\ \ \hbox{in}\ H^4(X\times X)\ ,\]
   where $R_0, R_2, R_4$ are cycles supported on $x\times X$, resp. on $D\times D$, resp. on $X\times x$.
 Applying the nilpotence theorem to the homologically trivial cycle
   \[ \Gamma-\Delta_X-R_0-R_2-R_4\ \ \in A^2(X\times X)\ ,\]
   and noting that the $R_i$ do not act on $A^2_{hom}(X)=A^2_{AJ}(X)$, we find there exists $N\in\NN$ such that
   \[   (\Gamma^{\circ N})_\ast=\hbox{id}\colon\ \ A^2_{hom}(X)\ \to\ A^2_{hom}(X)\ ,\]
   and we are done.
    \end{proof}

  %  \[  \Gamma-\Delta-\gamma=0\ \ \hbox{in}\ H^4(X\times X)\ ,\]
  %  where $\Delta$ is the diagonal and $\gamma$ is a cycle supported on $D\times D$, for some divisor $D\subset X$. The nilpotence theorem then implies there exists $N\in\NN$ such that
   % \[  \bigl(\Gamma-\Delta-\gamma\bigr)^{\circ N}=0\ \ \hbox{in}\ A^2(X\times X)_{\QQ}\ .\]  
  %  Developing this expression (and noting that the correspondence $\gamma$ can not act on $A^2_{hom}(X)_{\QQ}=A^2_{AJ}(X)_{\QQ}$), this gives
  %  \[  (\Gamma^{\circ N})_\ast=\hbox{id}\colon\ \ A^2_{hom}(X)_{\QQ}\ \to\ A^2_{hom}(X)_{\QQ}\ .\]
  %  It follows that $\Gamma_\ast$ is both injective and surjective. 
       %   \end{proof}
 
 Another consequence concerns a conjecture of Voevodsky's concerning smash--equivalence \cite{Voe}:
 
 \begin{corollary} Let $Z$ be a product $Z=X_1\times\cdots\times X_s$, where the $X_i$ are $K3$ surfaces as in theorem \ref{main}. Then smash--equivalence and numerical equivalence coincide for $1$--cycles on $Z$.
 \end{corollary}
 
 \begin{proof} As noted by Vial \cite[Theorem 3.17]{V3}, this is true for any variety for which the group of $0$--cycles is spanned (via the action of correspondences) by a product of curves.
\end{proof}
 
 A final consequence is that the Beauville--Voisin conjecture \cite[Conjecture 1.3]{V12} is true for the Hilbert schemes of $X$:
  
  \begin{corollary}\label{BV} Let $X$ be a $K3$ surface as in theorem \ref{main}, and let $X^{[m]}$ denote the Hilbert scheme of $m$ points on $X$. Then the restriction of the cycle class map $A^\ast(X^{[m]})_{\QQ}\to H^\ast(X^{[m]})$ to the $\QQ$--subalgebra generated by divisors and Chern classes of the tangent bundle is injective, for all $m\in\NN$.
  \end{corollary}
  
  \begin{proof} Yin has proven \cite{Yin} that the Beauville--Voisin conjecture is true for any $K3$ surface with finite--dimensional motive.
  \end{proof}

\end{document}